\documentclass[english,12pt]{article}
\usepackage[english]{babel}
\usepackage[normalem]{ulem}
\usepackage{amssymb,amsmath,amsthm,mathtools,braket,aligned-overset}
\usepackage{thmtools,thm-restate}

\usepackage{graphicx,xcolor}
\usepackage{float}
\usepackage{caption,subcaption}
\usepackage{tikz}
\usetikzlibrary{calc,decorations.pathmorphing,decorations.text,decorations.markings,matrix,shadings,shapes.geometric}

\usepackage{array,booktabs}
\usepackage{enumerate,enumitem}

\usepackage{algorithm}
\usepackage{algpseudocode}

\usepackage{appendix}
\usepackage{titlesec}
\usepackage{authblk}

\usepackage{hyperref}
\usepackage[capitalise,nameinlink,noabbrev]{cleveref}

\usepackage[square,sort,comma,numbers]{natbib}
\usepackage{etoolbox}

\AtBeginEnvironment{thebibliography}{%
  \setlength{\parskip}{0pt}%
  \setlength{\itemsep}{0pt}%
}

\allowdisplaybreaks

\theoremstyle{plain}
\newtheorem{thm}{Thm}[section]

\newtheorem{claim}{Claim}
\newtheorem{theorem}[thm]{Theorem}
\newtheorem{lemma}[thm]{Lemma}

\newtheorem{proposition}[thm]{Proposition}

\newtheorem{problem}[thm]{Problem}

\newenvironment{proof*}[1][Proof]
{\begin{proof}[#1]}
{\end{proof}}

\allowdisplaybreaks[4]

\begin{document}

\title{A local clique density theorem in $H$-free graphs}

\date{}
\author{Jiaao Li$^1$}
\author{Xinyuan Li$^1$}
\author{Yan Wang$^2$}
\author{Zhouningxin Wang$^1$}

\affil{\small $^1$School of Mathematical Sciences and LPMC, Nankai University, Tianjin 300071, China \linebreak
$^2$School of Mathematical Sciences, Shanghai Jiao Tong University, Shanghai 200240, China \linebreak Emails: lijiaao@nankai.edu.cn; xinyuanli@mail.nankai.edu.cn; yan.w@sjtu.edu.cn; wangzhou@nankai.edu.cn}

\maketitle

\begin{abstract}
In 2016, Reiher's clique density theorem
determined the minimum number of copies of $K_t$ in a graph with a prescribed edge density. In this paper, we investigate its local version and prove a local clique density theorem in $H$-free graphs as follows. For integers $r$ and $t$ with $2\leq t\leq r-1$, any $r$-chromatic graph $H$, any real numbers $\gamma$ and $\alpha$ with $\frac{t-2}{2(t-1)}\leq\gamma\leq \frac{r-2}{2(r-1)}$ and $0\leq\alpha\leq 1$, we determine the maximum value $\beta:=\beta(r,t,\alpha,\gamma)$ such that for every $n$-vertex $H$-free graph $G$ with at least $\gamma n^2$ edges, every $\lceil\alpha n\rceil$-vertex subset in $G$ contains at least $(\beta-o(1))n^{t}$ copies of $K_t$. In particular, when $H=K_r$, every $\lceil\alpha n\rceil$-vertex subset contains at least $\lfloor\beta n^t\rfloor$ copies of $K_t$, which is an exact bound. For suitable choices of $\alpha$ and $\gamma$, namely, those for which all part ratios in the corresponding extremal construction are rational, this bound is attained for infinitely many values of $n$.
\end{abstract}


\section{Introduction}
\label{sec:intro}

\subsection{Background and motivation on clique density}
Counting cliques or edges in graphs is a fundamental problem in graph theory. In 2016, Reiher~\cite{2016-reiher} counted the number of cliques in a graph with a prescribed lower bound on its number of edges.
\begin{theorem}[The clique density theorem~\cite{2016-reiher}]\label{thm-reiher's clique density thm}
    Let $t$ be an integer with $t\geq3$ and let $\gamma$ be a real number with $\gamma\in[0,\frac{1}{2})$. Every $n$-vertex graph with at least $\gamma n^2$ edges contains at least 
    \[
    \binom{s}{t}\ell^{t}n^t+\binom{s}{t-1}\ell^{t-1}(1-s\ell)n^t
    \]
    copies of $K_t$, where $s\geq 1$ is an integer such that $\gamma\in[\frac{s-1}{2s},\frac{s}{2(s+1)}]$ and $\ell\in [\frac{1}{s+1},\frac{1}{s}]$ is implicitly determined by $\gamma=\frac{s-1}{2s}(s\ell)^2+s\ell(1-s\ell)$.
\end{theorem}

Recently, the asymptotically sharp clique-to-clique density function was determined independently by Ma, Wang, and Zhu~\cite{MWZ2026}, and by Li, Liu, and Zhang~\cite{LLZ2026}: For every $2\le s<t$, they obtained the minimum possible number of $K_t$'s in a graph with a prescribed number of $K_s$'s. Their theorems extend Reiher's clique density theorem and give a complete global answer to the $K_s$-to-$K_t$ density problem. 

It is natural to consider this problem from a local perspective: Instead of estimating the total number of cliques in the whole graph, we ask for the minimum number of copies of $K_t$ contained in every $\alpha n$-vertex subset of a graph with a prescribed edge density. However, for general graphs, such a local version is essentially a consequence of the global clique density theorem and does not lead to new extremal phenomena. To see this, let $e_{t}(G)$ be the number of $K_t$'s in $G$. Assume that $G$ is an $n$-vertex graph with $\gamma n^2$ edges. For any subset $U\subseteq V(G)$ with $|U|=\alpha n$, the number of edges not contained in $U$ is at most $\binom{(1-\alpha)n}{2}+\alpha(1-\alpha)n^2$, and thus, the set $U$ spans at least $\gamma n^2-\binom{(1-\alpha)n}{2}-\alpha(1-\alpha)n^2$ edges. Then applying Reiher's clique density theorem, we can immediately obtain a lower bound on the number of copies of $K_t$ in the subgraph induced by $U$. Moreover, this bound is tight. Indeed, let $F$ be an $\alpha n$-vertex graph with $\gamma n^2-\binom{(1-\alpha)n}{2}-\alpha(1-\alpha)n^2$ edges and containing such a number of copies of $K_t$. For any graph $G$ that contains $F$ as an induced subgraph and contains all the other edges, every $\alpha n$-vertex subset of $G$ contains at least $e_{t}(F)$ copies of $K_{t}$ and the vertex subset $V(F)$ reaches this bound. Therefore, the local version of the clique density theorem for arbitrary graphs follows directly from the global one. This observation motivates us to consider the local density problem under additional structural constraints, in particular for $K_r$-free graphs.

\begin{problem}\label{prob-local density-inverse problem}
    Let $r$ and $t$ be integers with $2\leq t\leq r-1$, and let $\alpha, \gamma$ be real numbers with $0\leq\alpha\leq 1$ and $0\leq\gamma\leq\frac{r-2}{2(r-1)}$. Determine the largest $\beta:=\beta(r,t,\alpha,\gamma)$ such that, in every $n$-vertex $K_{r}$-free graph with at least $\gamma n^{2}$ edges, every $\lceil\alpha n\rceil$-vertex subset spans at least $\lfloor\beta n^{t}\rfloor$ copies of $K_t$. 
\end{problem}

Equivalently, \cref{prob-local density-inverse problem} aims to seek the maximum $\beta\neq 0$ such that if an $n$-vertex $K_{r}$-free graph $G$ contains a set of $\alpha n$ vertices spanning at most $\beta n^{t}$ copies of $K_t$, then $e(G)\leq \gamma n^{2}$.

In this paper, we resolve~\cref{prob-local density-inverse problem} completely.
The solution of~\cref{prob-local density-inverse problem} provides an alternative approach to finding sparse subsets in a $K_{r}$-free graph. To obtain a sparse subset of a desired size, it suffices first to locate a sparse subset whose size is substantially smaller than the target. This relaxation then leaves a larger remaining subgraph, thereby facilitating the discovery of other sparse subsets.

\medskip
Finally, we remark that \cref{prob-local density-inverse problem} is closely related to the local density problem proposed by Erd\H{o}s \cite{1975-erdos}. Indeed, \cref{prob-local density-inverse problem} asks to determine 
the maximum possible number of edges of the host graph when the size of a sparse subset is given, or equivalently, to determine the minimum number of edges (or the minimum number of copies of $K_t$) in every subset of the given size under the assumption that the host graph is dense. Therefore, a solution to \cref{prob-local density-inverse problem} provides a sufficient condition for the existence of sparse induced subgraphs, which is one of the central questions in the local density problem.

\begin{problem}[Erd\H{o}s' local density problem~\cite{1975-erdos}]\label{prob-erdos local density problem}
    Let $r$ be an integer with $r\geq 3$ and let $\alpha$ be a real number with $0\leq\alpha\leq 1$. Determine the minimum value $\theta:=\theta(r,\alpha)$ such that every $n$-vertex $K_{r}$-free graph contains a set of $\lfloor\alpha n \rfloor$ vertices spanning at most $\lceil\theta n^{2}\rceil$ edges.
\end{problem}
This problem generalizes the celebrated Tur\'an theorem; the case $\alpha=1$ is precisely Tur\'an's theorem. Many results toward this problem have been obtained; see \cite{1994-erdos, 1995-krivelevich, 1998-brandt, 1990-chung, 2003-sudakov, 2006-sudakov, 2015-norin, 2019-reiher, 2022-ma, 2021-razborov, 2023-reiher, 2024-balogh} for references.

The connection between the local density problem and its inverse counterpart can already be seen in the triangle-free case. If a triangle-free graph $G$ contains a large independent set, then the number of edges of $G$ cannot be too large. More precisely, we have the following warm-up proposition, which can be viewed as a basic example of the inverse local density phenomenon.
\begin{proposition}\label{prop-simple K3-free case}
    Assume that $G$ is an $n$-vertex $K_3$-free graph with an independent set $I$ of size at least $\frac{n}{2}$. Then $e(G)\leq |I|(n-|I|)$, with equality holding if and only if $G$ is isomorphic to the complete bipartite graph $K_{2}[|I|,n-|I|]$.
\end{proposition}

\subsection{Our results}\label{sec:contribution}
In this paper, we completely resolve \cref{prob-local density-inverse problem} for $K_r$-free graphs by determining $\beta(r,t,\alpha,\gamma)$ for all admissible values of the parameters. Moreover, we obtain an asymptotic extension to $H$-free graphs for every $r$-chromatic graph $H$.

Given an integer $r$ with $r\geq 3$ and a real number $\gamma$ with $0\leq\gamma\leq \frac{r-2}{2(r-1)}$, for each integer $s$ with $2\leq s\leq r$, let 
\[d_{r,\gamma}(s):=\sqrt{\frac{2(s-1)(r-s)}{r-1}\left(\frac{r-2}{2(r-1)}-\gamma \right)}.\]
We simply write $d(s)$ for $d_{r,\gamma}(s)$ when no confusion can arise.

Our first result gives an upper bound on the \emph{$K_{s}$-independence number}, defined as the maximum order of an induced $K_{s}$-free subgraph, of a $K_{r}$-free graph with at least $\gamma n^{2}$ edges.
\begin{theorem}\label{thm-upper bound on Ks-independence number}
    Let $r$ and $s$ be integers with $r\geq 3$ and $2\leq s\leq r$, and let $\gamma$ be a real number such that $0\leq\gamma\leq\frac{r-2}{2(r-1)}$.  
    If $G$ is an $n$-vertex $K_{r}$-free graph with at least $\gamma n^{2}$ edges, then the $K_{s}$-independence number of $G$ is at most $(\frac{s-1}{r-1}+d(s))n$.
\end{theorem}

To state the solution to~\cref{prob-local density-inverse problem} in its exact form, we first introduce the extremal examples, which clarify how $\beta$ is obtained.
The constants $r,s,$ and $\gamma$ are defined as above. Let $t$ be an integer with $2\leq t\leq r-1$ and let $\alpha$ be a real number with $0\leq\alpha\leq 1$. We distinguish the following three extremal cases based on the range of $\alpha$.

\begin{enumerate}
    \item For $0\leq \alpha\leq \frac{t-1}{r-1}+d(t)$, we consider the complete $(r-1)$-partite graph with its $t-1$ parts of size $\frac{n}{r-1}+\frac{d(t)}{t-1}n$ and the remaining $r-t$ parts of size $\frac{1}{r-t}(\frac{r-t}{r-1}-d(t))n$. Clearly, this $(r-1)$-partite graph contains $\alpha n$ vertices (contained in the union of the largest $t-1$ parts) spanning no $K_{t}$. Hence, the corresponding $\beta$ equals $0$.

    \item For $t\leq s\leq r-1$ and $\frac{s-1}{r-1}+d(s)<\alpha\leq \frac{s}{r-1}+\frac{r-s-1}{r-s}d(s)$, we consider the complete $(r-1)$-partite graph with $s-1$ parts having size $\frac{n}{r-1}+\frac{d(s)}{s-1}n$ and the remaining $r-s$ parts having size $\frac{1}{r-s}(\frac{r-s}{r-1}-d(s))n$. The sparsest $\alpha n$-vertex subset is the union of all the largest $s-1$ parts together with a $(\alpha-\frac{s-1}{r-1}-d(s))n$-vertex subset of one of the remaining parts. Such a (sparsest) $\alpha n$-vertex subset induces 
{\small \[
\binom{s-1}{t}\left(\frac{1}{r-1}+\frac{d(s)}{s-1}\right)^t n^t+\binom{s-1}{t-1}\left(\frac{1}{r-1}+\frac{d(s)}{s-1}\right)^{t-1}\left(\alpha-\frac{s-1}{r-1}-d(s)\right)n^t
\]}
copies of $K_t$. By convention, we define $\binom{s-1}{t} = 0$ when $t = s$.

\item For $t\leq s\leq r-2$ and $\frac{s}{r-1}+\frac{r-s-1}{r-s}d(s)<\alpha\leq \frac{s}{r-1}+d(s+1)$, let $L$ be the unique solution of the function
\begin{equation}\label{equ:L}
\binom{s-1}{2}L^2+(s-1)L(\alpha-(s-1)L)=\gamma-\alpha(1-\alpha)-\frac{r-s-2}{2(r-s-1)}(1-\alpha)^2
\end{equation}
within the range
\begin{equation}\label{equ:L2}
\frac{s-1}{s}\left(\frac{s}{r-1}+d(s+1)\right)\leq (s-1)L\leq \frac{s-1}{r-1}+d(s).
\end{equation}
Then we consider the complete $(r-1)$-partite graph with $s-1$ parts having size $Ln$, a part having size $(\alpha-(s-1)L)n$, and the remaining $r-s-1$ parts having size $\frac{1}{r-s-1}(1-\alpha)n$. The sparsest $\alpha n$-vertex subset is the union of all the largest $s-1$ parts and the part of size $(\alpha-(s-1)L)n$. Such a sparsest $\alpha n$-vertex subset spans 
\[
\binom{s-1}{t}L^t n^t+\binom{s-1}{t-1}L^{t-1}\left(\alpha-(s-1)L\right)n^t
\]
copies of $K_t$.
\end{enumerate}

We shall show that the optimal value of $\beta$ is attained by these $(r-1)$-partite graphs. Consequently, we define the exact value of $\beta$ as follows.

For $t=2$, which corresponds to the case of $K_2$, we define the function $\beta$ in terms of $r,\alpha,$ and $\gamma$ as follows. Let
\begin{small}
\begin{align*}
    \beta(r,2,\alpha,\gamma)=\left\{
        \begin{array}{cc}
        0, &~\text{if}~~ 0\leq\alpha\leq \frac{1}{r-1}+d(2),\\
        \gamma-(1-\alpha)(\frac{r-2}{r-1}+d(r-1)), &~\text{if}~~ \frac{r-2}{r-1}+d(r-1)\leq \alpha\leq1,
        \end{array}
        \right.  
\end{align*}
\end{small}
and for each integer $s$ with $2\leq s\leq r-2$, 
{\small \begin{align*}
\beta(r,2,\alpha,\gamma)
=
\begin{cases}
\displaystyle
\frac{s-2}{2(s-1)}
\left(\frac{s-1}{r-1}+d(s)\right)^2
+
\left(\frac{s-1}{r-1}+d(s)\right)
\left(\alpha-\frac{s-1}{r-1}-d(s)\right),
\\[2ex]
\qquad\qquad\qquad\qquad\text{if }
\dfrac{s-1}{r-1}+d(s)<\alpha
\leq
\dfrac{s}{r-1}+\dfrac{r-s-1}{r-s}d(s),
\\[2.4ex]
\displaystyle
\gamma-\frac{r-s-2}{2(r-s-1)}
(1-\alpha)^2-\alpha(1-\alpha),
\\[2ex]
\qquad\qquad\qquad\qquad\text{if }
\dfrac{s}{r-1}+\dfrac{r-s-1}{r-s}d(s)
<
\alpha
\leq
\dfrac{s}{r-1}+d(s+1).
\end{cases}
\end{align*}}

For any $t\geq 2$, we define the function $\beta$ in terms of $r,t,\alpha,$ and $\gamma$ as follows. Let $\beta(r,t,\alpha,\gamma)=0$ if $0\leq\alpha\leq \frac{t-1}{r-1}+d(t)$. 
When $t\leq s\leq r-1$ and $\frac{s-1}{r-1}+d(s)<\alpha\leq \frac{s}{r-1}+\frac{r-s-1}{r-s}d(s)$, let
{\small \[
\beta(r,t,\alpha,\gamma)=\binom{s-1} {t}\left(\frac{1}{r-1}+\frac{d(s)}{s-1}\right)^t+\binom{s-1}{t-1}\left(\frac{1}{r-1}+\frac{d(s)}{s-1}\right)^{t-1}\left(\alpha-\frac{s-1}{r-1}-d(s)\right).
\]}
When $t\leq s\leq r-2$ and $\frac{s}{r-1}+\frac{r-s-1}{r-s}d(s)<\alpha\leq \frac{s}{r-1}+d(s+1)$, let
\[
\beta(r,t,\alpha,\gamma)=\binom{s-1} {t}L^t+\binom{s-1}{t-1}L^{t-1}(\alpha-(s-1)L),
\]
where $L$ is as defined above (see~\eqref{equ:L} and~\eqref{equ:L2}).

We are now ready to present the following theorem, which resolves \cref{prob-local density-inverse problem} and is the main result of this paper.

\begin{theorem}\label{thm-inverse problem of local density}

    Let $r$ and $t$ be integers with $2\leq t\leq r-1$, and let $\alpha, \gamma$ be real numbers with $0\leq\alpha\leq 1$ and $0\leq\gamma\leq\frac{r-2}{2(r-1)}$. Let $\beta=\beta(r,t,\alpha,\gamma)$ be defined as above. If $G$ is an $n$-vertex $K_{r}$-free graph with at least $\gamma n^{2}$ edges, then every $\lceil\alpha n\rceil$-vertex subset of $G$ spans at least $\lfloor\beta n^{t}\rfloor$ copies of $K_{t}$. 
\end{theorem}

Equivalently, taking the contrapositive of~\cref{thm-inverse problem of local density}, if $\beta\not=0$ and $G$ contains a $\lceil\alpha n\rceil$-vertex subset spanning fewer than $\lfloor\beta n^{t}\rfloor$ copies of $K_{t}$, then $e(G)\leq \gamma n^{2}$. Moreover, this result is tight up to an error term $o(n^t)$. The almost extremal cases are characterized as above, and the error term $o(n^t)$ comes from the fact that the complete $(r-1)$-partite graph constructed above has real-number-weighted order in each part. For every pair of rational numbers $\alpha$ and $\gamma$, when the associated parameters $d(s)$, $L$, and $\beta$ satisfy the required integrality conditions, ensuring that every part in the extremal construction has integer size, there exist infinitely many integers $n$ and examples for which the bound is attained.

We remark here that it suffices to restrict ourselves to the case when $t=2$ in the proof of \cref{thm-inverse problem of local density}. Indeed, the case when $t=2$ asserts precisely the desired local edge lower bound: for every $U\subseteq V(G)$ with $|U|=\alpha n$, $e(G[U])\geq \beta(r,2,\alpha,\gamma)n^2=\frac{\beta(r,2,\alpha,\gamma)}{\alpha^2}|U|^2.$
We therefore apply Reiher's clique density theorem (\cref{thm-reiher's clique density thm}) to the graph $G[U]$. The parameters in the definition of $\beta(r,t,\alpha,\gamma)$ are chosen exactly so that the clique density lower bound obtained in this way is
$\beta(r,t,\alpha,\gamma)n^t$.
Consequently, once the case when $t=2$ has been established, the assertion of \cref{thm-inverse problem of local density} for every $t\geq 2$ follows immediately from Reiher's theorem.

Finally, the Regularity Lemma~\cite{1978-szemeredi} enables us to extend \cref{thm-inverse problem of local density} to $H$-free graphs, leading to the following theorem.

\begin{theorem}\label{thm-H-free graph case}
Let $r$ and $t$ be integers with $2\leq t\leq r-1$, and let $\alpha, \gamma$ be real numbers with $0\leq\alpha\leq 1$ and $0\leq\gamma\leq\frac{r-2}{2(r-1)}$. Let $\beta=\beta(r,t,\alpha,\gamma)$ be defined as above. Assume $H$ is an $r$-chromatic graph. If $G$ is an $n$-vertex $H$-free graph with at least $\gamma n^{2}$ edges, then every set of $\alpha n$ vertices in $G$ spans at least $(\beta-o(1))n^{t}$ copies of $K_{t}$.  
\end{theorem}

Before proceeding to the proofs of \cref{thm-upper bound on Ks-independence number} and \cref{thm-inverse problem of local density}, we make the following remarks. Assume $G$ is an $n$-vertex $K_r$-free graph with at least $\gamma n^2$ edges.
\begin{enumerate}
    \item We can assume that $\alpha$ and $\gamma$ are rational. Assume that $\alpha$ is irrational. Then there exists a sequence $\{\alpha_k\}_{k=1}^{+\infty}$ of decreasing rational numbers such that $\alpha_k\to \alpha$. When $k$ is sufficiently large, we have $\lceil\alpha n\rceil=\lceil\alpha_{k} n\rceil$ and $\lfloor \beta(r,t,\alpha,\gamma) n^t\rfloor\leq\lfloor \beta(r,t,\alpha_k,\gamma) n^t\rfloor$ since $\alpha n$ is irrational and $\beta(r,t,\alpha,\gamma)$ is a continuous monotone increasing function. This implies we only need to consider the case when $\alpha$ is rational, which ensures that, for sufficiently large $k$, every $\lceil\alpha_{k} n\rceil$-vertex subset in $G$ contains at least $\lfloor \beta(r,t,\alpha_k,\gamma) n^t\rfloor$ copies of $K_t$, as desired. The case when $\gamma$ is rational can be proved similarly.

    \item We can further assume that $\alpha n$ and $\gamma n^2$ are integers. Otherwise, we consider a blow-up $G[N]$ of $G$ such that $\alpha n N$ and $\gamma n^2 N^2$ are integers. If every $\alpha n N$-vertex subset of $G[N]$ spans at least $\lfloor\beta n^t N^t \rfloor$ copies of $K_t$, then it follows from $\lceil \alpha n\rceil N\geq \alpha n N$ and $\lfloor\beta n^t\rfloor N^t\leq \lfloor\beta n^t N^t \rfloor$ that every $\lceil \alpha n\rceil N$-vertex subset of $G[N]$ spans at least $\lfloor\beta n^t\rfloor N^t$ copies of $K_t$, and thus every $\lceil \alpha n\rceil$-vertex subset of $G$ contains at least $\lfloor\beta n^t\rfloor$ copies of $K_t$.
\end{enumerate}
By these observations, we may assume that $\alpha n$ and $\gamma n^2$ are integers throughout this paper.

\section[MainProof]{Proofs of \cref*{thm-upper bound on Ks-independence number} and \cref*{thm-inverse problem of local density}}
\label{section-mainproof}

For any subset $U\subseteq V(G)$, let $G[U]$ denote the subgraph induced by $U$ of $G$. For two disjoint subsets $X, Y\subseteq V(G)$, we use $E_G(X,Y)$ (or simply $E(X,Y)$) to denote the set of edges between $X$ and $Y$ in $G$ and let $e(X,Y)=|E(X,Y)|$. We denote by $K_k[a_1,\ldots,a_k]$ the complete $k$-partite graph whose $i$-th part has size $a_i$ for each $i \in [k]$.

In this section, we present the proofs of \cref{thm-upper bound on Ks-independence number} and \cref{thm-inverse problem of local density}, starting with several lemmas.

\subsection[EdgesKr-freeGraphs]{An upper bound on the number of edges of $K_r$-free graphs}

In this subsection, we first give an upper bound on the number of crossing edges between a clique and a complete $(r-1)$-partite graph. 

\begin{lemma}\label{lem-two cliques}
Let $r$ and $t$ be two integers such that $1\leq t\leq r-1$. Let $G$ be a $K_{r}$-free graph. Assume that $G$ admits a vertex partition $V(G)=V_0\cup V_1$ such that $|V_0|=t$, $G[V_0]$ is isomorphic to $K_t$, and $G[V_1]$ is isomorphic to a complete $(r-1)$-partite graph $K_{r-1}[a_{1},\dots,a_{r-1}]$ with $a_{1}\geq \dots\geq a_{r-1}$. Then we have $$e(V_0,V_1)\leq t(a_{1}+\dots+a_{r-1})-(a_{r-t}+\dots+a_{r-1}).$$
\end{lemma}
\begin{proof}
We first observe that the number of edges of the complete bipartite graph $K_2[t, n-t]$ is $t(a_1 + \dots + a_{r-1})$. We compare the edge set $E(V_0, V_1)$ with that of $K_2[t, n-t]$. We claim that if $G$ is $K_r$-free, then at least $a_{r-t} + \dots + a_{r-1}$ edges are missing from $E(V_0,V_1)$, compared to $K_2[t, n-t]$. Equivalently, in $G[V_1] = K_{r-1}[a_1,\dots,a_{r-1}]$, there are at least $t$ parts in which every vertex misses at least one neighbor in $G[V_0]$. Otherwise, there would be at most $t-1$ such parts, which means that in at least $r-t$ parts there exists some vertex adjacent to all vertices of $K_t$. This would form a $K_r$, contradicting the assumption that $G$ is $K_r$-free.
\end{proof}

The following lemma provides an estimate of the number of edges of $K_r$-free graphs based on a greedy partition strategy. Such a partition is also useful for solving other problems; see \cite{2025-HMWZ} and references therein for instances. 

\begin{lemma}\label{lem-partition implies bound on edges}
Assume $G$ is a $K_{r}$-free graph with a partition $V(G)=V_{1}\cup \dots\cup V_{r-1}$ such that, for each $i\in [r-1]$, $V_{i}$ can be partitioned into $V_{i1}, \ldots, V_{is_{i}}$ such that each $G[V_{ij}]$, for $j\in [s_i]$, is isomorphic to $K_i$, and $G[\bigcup\limits_{\ell=1}^i V_{\ell}]$ contains no $K_{i+1}$. For each $i\in [r-1]$, let $b_{i}:=\sum\limits_{j=r-i}^{r-1}\frac{|V_{j}|}{j}$. Then $$e(G)\leq e(K_{r-1}[b_{1},\dots,b_{r-1}]).$$
\end{lemma}
\begin{proof}
    Observe that each $G[V_i]$ is obtained by greedily packing a maximum number of vertex-disjoint copies of $K_i$ in $G\setminus(\bigcup_{\ell=i+1}^{r-1}V_\ell).$
    For distinct $i,j$ with $1\leq i<j\leq r-1$ and $s\in [s_i], t\in [s_j]$, we apply \cref{lem-two cliques} to $G[V_{is}]$ (which is isomorphic to $K_i$) and $G[V_{jt}]$ (which is isomorphic to $K_j$) with viewing $K_j=K_j[1,1,\ldots, 1]$ and thus $e(V_{is},V_{jt})\leq ij-i$. Hence, for distinct $i,j$ with $1\leq i<j\leq r-1$, we have that
    \begin{equation}\label{ineq-}
        e(V_{i},V_{j})= \sum_{p=1}^{s_p}\sum_{q=1}^{s_q}e(V_{ip},V_{jq})\leq \sum_{p=1}^{s_p}\sum_{q=1}^{s_q}i(j-1)= \frac{|V_{i}||V_{j}|}{ij}i(j-1).
    \end{equation}
    It follows from Tur\'{a}n's theorem that $e(V_{i})\leq \binom{i}{2}(\frac{|V_{i}|}{i})^{2}$.
    Now we derive
    \begin{align*}
        e(G)=\sum_{i=1}^{r-1}e(V_{i})+\sum_{1\leq i<j\leq r-1}e(V_{i},V_{j})&=\sum_{i=1}^{r-1}e(V_i)+\sum_{1\leq i\leq j\leq r-1}\sum_{s=1}^{s_i}\sum_{t=1}^{s_j}e(V_{is},V_{jt})\\
        &\leq \sum_{i=1}^{r-1}\binom{i}{2}\left(\frac{|V_{i}|}{i}\right)^{2}+\sum_{1\leq i<j\leq r-1}\frac{|V_{i}||V_{j}|}{ij}i(j-1)\\
        &=\sum_{1\leq k<\ell\leq r-1}\left(\sum_{i=r-k}^{r-1}\frac{|V_{i}|}{i}\right)\left(\sum_{j=r-\ell}^{r-1}\frac{|V_{j}|}{j}\right)\\
        &=e(K_{r-1}[b_{1},\dots,b_{r-1}]),
    \end{align*}
    where $b_{i}=\sum\limits_{j=r-i}^{r-1}\frac{|V_{j}|}{j}$. 
    Indeed, the inequality in the second line follows from Tur\'{a}n's theorem and \eqref{ineq-}. 
    The equality in the third line follows from double-counting on the number of edges of $K_{r-1}[b_1,\dots,b_{r-1}]$, and the last equality is immediate from the definition.   
\end{proof}

\subsection[Ks independence number]{An upper bound on the $K_{s}$-independence number}

In this subsection, we shall prove \cref{thm-upper bound on Ks-independence number}. Let $r$ and $s$ be two integers with $r\geq s$. We first introduce a symmetrization operation $Sym_{r,s}(G,H)$, which is motivated by the method notably used in Zykov's proof of Tur\'{a}n's theorem \cite{1949-zykov}.
Symmetrization techniques enable the reduction of many extremal problems to specific, well-structured graph classes. For example, Zykov's proof of Turán's theorem \cite{1949-zykov} reduces the problem of determining the maximum number of edges of a $K_r$-free graph to that of counting edges of a complete balanced $(r-1)$-partite graph; In \cite{2022-reiher,2025-reiher}, the Ramsey–Turán problem for triangle-free graphs is similarly reduced to a class of graphs whose edge counts are readily estimable. The central idea underlying our proofs of \cref{thm-upper bound on Ks-independence number} and \cref{thm-inverse problem of local density} is to employ a refined symmetrization argument, thereby reducing the problems to $(r-1)$-partite graphs, from which the desired exact bounds can be derived.

\begin{algorithm}
\caption{Symmetrization $Sym_{r,s}(G,H)$}
\label{alg:symmetrization}

\begin{algorithmic}[1]

\Statex \hspace{-\algorithmicindent}\textbf{Input:} A $K_r$-free graph $G$ with a $K_s$-free subgraph $H$
\Statex \hspace{-\algorithmicindent}\textbf{Output:} A graph pair $(G',H')$

\State Set $R_0=V(H)$ and $G_0=G$. Let $i\gets 1$. 
\While{$R_{i-1}\neq\emptyset$} 
\State Choose $u_i\in R_{i-1}$ with maximum degree $d_{G_{i-1}}(u_i)$. 
\State Define $S_i:=R_{i-1}\setminus N(u_i)$ and $R_i:=R_{i-1}\setminus S_i=N_{G[R_{i-1}]}(u_i)$. 
\State Obtain $G_i$ from $G_{i-1}$ by deleting all edges incident with $S_i$ and adding all edges between $S_i$ and $N_{G_{i-1}}(u_i)$. 
\State Set $i\gets i+1$. 
\EndWhile

\State \Return $(G',H')=(G_{i-1},G_{i-1}[V(H)])$
\end{algorithmic}
\end{algorithm}

The input of this operation is a pair of graphs $(G, H)$, where $G$ is a $K_{r}$-free graph and $H$ is a $K_{s}$-free subgraph of $G$. Its output is also a pair of graphs $(G',H')$, i.e., $Sym_{r,s}(G,H)=(G',H')$, satisfying that $G'$ is still $K_{r}$-free with $e(G')\geq e(G)$, $H'$ is complete $(s-1)$-partite and $G[V(G)\setminus V(H)]$ is isomorphic to $G'[V(G)\setminus V(H)]$; see~\cref{alg:symmetrization}. The Zykov's symmetrization of $G$ can be regarded as $Sym_{r,r}(G,G)$ (i.e, $s=r$ and $H=G$), which we abbreviate as $Sym_r(G)$ or $Sym(G)$ when it is clear from the context. That is to say, $Sym(G)$ is a complete $(r-1)$-partite graph with $e(Sym(G))\geq e(G)$. 

This algorithm must terminate after at most $s-1$ steps since $H$ is $K_{s}$-free and at the $i$-th  step $H_{i}$ is the common neighborhood of a $K_{i}$ in $H$. One can also verify that $G_{i}$ is $K_{r}$-free and $e(G_{i})\geq e(G_{i-1})$. Clearly, the output $H'$ is complete $(s-1)$-partite. Furthermore, for each $i\in \{1,\dots,s-1\}$, all the vertices in $S_i$ have the same neighborhood. Moreover, it is clear that the edges inside $G[V(G) \setminus V(H)]$ remain unchanged and thus $G[V(G) \setminus V(H)]$ is isomorphic to $G'[V(G) \setminus V(H)]$.

Recall that $T_{k}(n)$ denotes the Tur\'{a}n graph on $n$ vertices with $k$ parts, i.e., $T_k(n) = K_k[a_1, \dots, a_k]$, where $|a_i - a_j| \leq 1$ for all $i, j \in \{1,\dots,k\}$. For two graphs $G_1$ and $G_2$, let $G_1 \times G_2$ denote the graph with the vertex set $V(G_1) \cup V(G_2)$, where $V(G_1)$ and $V(G_2)$ are assumed to be disjoint, and the edge set $E(G_1 \times G_2):=E(G_1) \cup E(G_2) \cup \{uv \mid u \in V(G_1),~ v \in V(G_2) \}$. Now we are ready to prove~\cref{thm-upper bound on Ks-independence number}.

\begin{proof}[Proof of \cref*{thm-upper bound on Ks-independence number}]
    Let $\mathcal{G}$ denote the set of $n$-vertex $K_{r}$-free graphs $G$ with $K_{s}$-independence number at least $(\frac{s-1}{r-1}+d(s)) n$ and with the maximum number of edges. Now it suffices to prove that every graph in $\mathcal{G}$ has at most $\gamma n^{2}$ edges.
    
    We shall first prove that, by symmetrization, there is a complete $(r-1)$-partite graph in $\mathcal{G}$. Let $G\in \mathcal{G}$ and let $F$ be a maximum $K_{s}$-free subgraph of $G$. By the definition, $|F|\geq (\frac{s-1}{r-1}+d(s)) n$. 
    We operate the symmetrization on $(G, F)$. Let $(G',F'):=Sym_{r,s}(G,F)$. We know that $G'$ is also a $K_{r}$-free graph, $F'$ is a $K_{s}$-free subgraph of $G'$ and it is complete $(s-1)$-partite, $e(G')\geq e(G)$ and $G[V(G)\setminus V(F)]=G'[V(G)\setminus V(F)]$. Then let $G''=Sym(G')$. It follows that $G''$ is complete $(r-1)$-partite.
    We claim that $G''$ is the desired graph belong to $\mathcal{G}$ and it suffices to verify $G''[V(F)]$ is still $K_{s}$-free. Note that $F'$ is complete $(s-1)$-partite and we denote those $(s-1)$ parts by $F_{1}, \ldots,  F_{s-1}$, respectively. Note that some of the $F_{i}$ may be empty. By the property of symmetrization, for each $i\in \{1,\dots,s-1\}$, all the vertices in $F_{i}$ have the same neighbors in $G'$. At the $i$-th step of $Sym(G')$, the identified vertex $u_i$ satisfies that for each $j\in \{1,\dots,s-1\}$, either $F_{j}\subseteq N(u_{i})$ or $F_{j}\cap N(u_{i})=\emptyset$. In the both cases $F_{j}$ remains independent after $i$-th step of $Sym(G')$, and thus $G''[V(F)]$ remains $(s-1)$-partite and $K_{s}$-free. Hence, we find a complete $(r-1)$-partite graph $G''\in\mathcal{G}$.

    We next prove that there is a graph $G^*$ in $\mathcal{G}$ isomorphic to $T_{s-1}\left((\frac{s-1}{r-1}+d(s))n\right)\times T_{r-s}\left((\frac{r-s}{r-1}-d(s))n\right)$. Observe that a complete $(r-1)$-partite graph $G^*$ belongs to $\mathcal{G}$, only if the size of the union of largest $s-1$ parts is at least $(\frac{s-1}{r-1}+d(s))n$. And it contains more edges when difference between the size of two parts is smaller. Let $\alpha=\frac{s-1}{r-1}+d(s)$. Hence each of largest $s-1$ parts has size $\frac{\alpha n}{s-1}$, and each of the remaining parts has size $\frac{(1-\alpha )n}{r-s}$.  Since
    \[
    \frac{s-2}{2(s-1)}\alpha^{2}+\alpha(1-\alpha)+\frac{r-s-1}{2(r-s)}(1-\alpha)^{2}=\gamma,
    \] 
    $e(G^*)=\gamma n^2$, and we are done.
    We know that in $\mathcal{G}$ every graph has at most $\gamma n^{2}$ edges as asserted. 
\end{proof}

\subsection[InverseProblem]{Proof of \cref*{thm-inverse problem of local density}}

\begin{proof}[Proof of \cref*{thm-inverse problem of local density}]
    As remarked immediately following the statement of \cref{thm-inverse problem of local density} in \cref{sec:contribution}, we only need to consider the case when $t=2$.

    Let $r\geq 3$, $0\leq\alpha\leq 1$, $0\leq\gamma\leq\frac{r-2}{2(r-1)}$ and $\beta=\beta(r, 2,\alpha, \gamma)$ be defined as above~\cref{thm-inverse problem of local density}. Assume $\beta \neq 0$. Let $\mathcal{G}$ denote the family of $n$-vertex $K_r$-free graphs $G$ that each contains a subset of $\alpha n$ vertices spanning at most $\lfloor \beta n^{2}\rfloor$ edges and, in addition, $G$ has the maximum number of edges. 

    \begin{claim}\label{claim:r-partite graphs}
    $\mathcal{G}$ contains an $(r-1)$-partite graph.
    \end{claim}
    
    \begin{proof*}[Proof of~\cref*{claim:r-partite graphs}]
    Let $G\in \mathcal{G}$ and let $S$ be a subset of $\alpha n$ vertices spanning at most $\lfloor \beta n^{2}\rfloor$ edges in $G$. Let $R_S=V(G)\setminus S$.  

    For each pair $(G, S)$, we claim that we can apply the symmetrization such that $G[R_S]$ to be $(r-1)$-partite. We first apply the symmetrization to make $G[R]$ to be $(r-1)$-partite. Let $(G',R'):=Sym_{r,r}(G,R)$. By definition, $G'[R']$ is complete $(r-1)$-partite, and then each part of $G'[R']$ is denoted by $R_{1},\dots,R_{r-1}$, respectively. We further assume that for each $i\in [r-1]$, $|R_{i}|=a_{i}$ with $a_{1}\geq \dots \geq a_{r-1}$ and note that some $R_{i}$'s might be empty. 
    
    After symmetrization, we have that $G[S]$ is isomorphic to $G'[S]$. Now we partition $S$ into $V_{1},\dots,V_{r-1}$, such that for each $i\in [r-1]$, $V_{i}$ can be partitioned into $V_{i1}\cup \cdots\cup V_{i{s_i}}$ such that each $G'[V_{ij}]$ (for $j\in [s_i]$) is isomorphic to $K_i$, and $G'[\bigcup\limits_{\ell=1}^i V_{\ell}]$ contains no $K_{i+1}$. Such a partition exists by greedily selecting vertex-disjoint copies of $K_r$, $K_{r-1}$, $\ldots$, and finally $K_1$, and grouping the corresponding vertex subsets accordingly. By~\cref{lem-partition implies bound on edges}, $$e_{G}(S)=e_{G'}(S)\leq \sum_{1\leq s<t\leq r-1}\left(\sum_{i=r-s}^{r-1}\frac{|V_{i}|}{i}\right)\left(\sum_{j=r-t}^{r-1}\frac{|V_{j}|}{j}\right).$$
    
    For each $i\in [r-1]$ and each $G'[V_{ij}]$ (isomorphic to $K_i$) in $ V_{i}$, by~\cref{lem-two cliques} we have $e_{G'}(V_{ij}, R)\leq i(a_{1}+\dots+a_{r-1})-(a_{r-i}+\dots+a_{r-1})$. 
    Now we divide the proof into two cases according to whether the upper bound of $e_{G'}(S)$ is larger than $\lfloor \beta n^{2}\rfloor$.

    \medskip
    \noindent\textbf{Case 1.} \emph{Assume that $\sum\limits_{1\leq s<t\leq r-1}\Big(\sum\limits_{i=r-s}^{r-1}\frac{|V_{i}|}{i}\Big)\Big(\sum\limits_{j=r-t}^{r-1}\frac{|V_{j}|}{j}\Big)\leq \lfloor \beta n^{2}\rfloor$.}
    \medskip
    
    We consider the following complete $(r-1)$-partite graph
    {\small $$
    G'':=K_{r-1}\left[a_{1}+\frac{|V_{r-1}|}{r-1},\dots,\, a_{i}+\sum_{k=r-i}^{r-1}\frac{|V_{k}|}{k},\dots,\, a_{r-1}+\sum_{k=1}^{r-1}\frac{|V_{k}|}{k}\right]. 
    $$ }
    Note that $G''$ contains a complete $(r-1)$-partite induced subgraph isomorphic to 
    $K_{r-1}\Big[\frac{|V_{r-1}|}{r-1},\dots,\sum\limits_{k=r-i}^{r-1}\frac{|V_{k}|}{k},\dots, \sum\limits_{k=1}^{r-1}\frac{|V_{k}|}{k}\Big],$
    denoted by $G''[S'']$. 
    Clearly, $|S''|=|V_{1}|+\dots+|V_{r-1}|=|S|=\alpha n$, and $e(S'')\leq \lfloor \beta n^{2}\rfloor$. To prove $G''\in \mathcal{G}$, it suffices to show $e(G'')\geq e(G)$. By the symmetrization, and combining \cref{lem-two cliques,lem-partition implies bound on edges}, we have that
    \begin{align*}
        e(G)\leq e(G')&= e(S)+e(R')+e(S,R')\\
        &=e(S)+e(R')+\sum_{i=1}^{r-1}e(V_{i},R')\\
        &\leq \sum_{1\leq s<t\leq r-1}\left(\sum_{i=r-s}^{r-1}\frac{|V_{i}|}{i}\right)\left(\sum_{j=r-t}^{r-1}\frac{|V_{j}|}{j}\right)+\sum_{1\leq s<t\leq r-1}a_s a_t\\
        &\qquad\qquad\qquad+\sum_{i=1}^{r-1}\frac{|V_{i}|}{i}(i(a_{1}+\dots+a_{r-1})-(a_{r-i}+\dots+a_{r-1})) \\
        &=\sum_{1\leq s<t\leq r-1}\left(a_{s}+\sum_{k=r-s}^{r-1}\frac{|V_{k}|}{k}\right)\left(a_{t}+\sum_{k=r-t}^{r-1}\frac{|V_{k}|}{k}\right)=e(G''),
    \end{align*}
    where the validity of the second last equality is similar to the proof of~\cref{lem-partition implies bound on edges}. It implies that $e(G'')\geq e(G)$. Thus we find a desired complete $(r-1)$-partite graph $G''$ belong to $\mathcal{G}$. This completes the proof of Case 1.

    \medskip
    \noindent\textbf{Case 2.}  \emph{Assume that $\sum\limits_{1\leq s<t\leq r-1}\Big(\sum\limits_{i=r-s}^{r-1}\frac{|V_{i}|}{i}\Big)\Big(\sum\limits_{j=r-t}^{r-1}\frac{|V_{j}|}{j}\Big) \geq \lfloor \beta n^{2}\rfloor$.} 
    \medskip
    
    We consider a $(r-1)$-partite graph $G''$ satisfying that
    $$
    G''\subseteq K_{r-1}\left[a_{1}+\frac{|V_{r-1}|}{r-1},\dots,a_{i}+\sum_{k=r-i}^{r-1}\frac{|V_{k}|}{k},\dots,a_{r-1}+\sum_{k=1}^{r-1}\frac{|V_{k}|}{k}\right],
    $$
    $G''$ has the same vertex set as this complete $(r-1)$-partite graph, contains a $(r-1)$-partite induced subgraph, denoted  $G''[S'']$, of $K_{r-1}\Big[\frac{|V_{r-1}|}{r-1},\dots,\sum\limits_{k=r-i}^{r-1}\frac{|V_{k}|}{k},\dots, \sum\limits_{k=1}^{r-1}\frac{|V_{k}|}{k}\Big]$ with $e(S'')= \lfloor \beta n^{2}\rfloor$, and furthermore, contains all the other edges incident to any vertex of $V(G'')\setminus S''$.
    
    Similar to Case 1, we have $|S''|=\alpha n$, $G''\supseteq S''$, $e(S'')=\lfloor \beta n^{2}\rfloor$, and
    {\small
    \begin{align*}
        e(G)\leq e(G')&= e(S)+e(R')+e(S,R')\\
        &=e(S)+e(R')+\sum_{i=1}^{r-1}e(V_{i},R')\\
        &\leq \lfloor \beta n^{2}\rfloor+\sum_{1\leq s<t\leq r-1}a_{i}a_{j}+\sum_{i=1}^{r-1}\frac{|V_{i}|}{i}(i(a_{1}+\dots+a_{r-1})-(a_{r-i}+\dots+a_{r-1})) \\
        &=e(G'').
    \end{align*}
    }
    So we have $e(G'')\geq e(G)$ as desired.
    This completes the proof of Case 2.
    \end{proof*}

    We have proved in~\cref{claim:r-partite graphs} that there is an $(r-1)$-partite graph in $\mathcal{G}$. This allows us to restrict our consideration to the subclass $\mathcal{G}^*\subseteq \mathcal{G}$ which consists of all $n$-vertex $(r-1)$-partite graphs $H$.
    For integers $c_1, \ldots, c_{r-1}$ such that $c_{1}+\dots+c_{r-1}=n$ and $c_{1}\geq c_{2}\geq \dots\geq c_{r-1}$, assume that $H\subseteq K_{r-1}[c_{1},\dots,c_{r-1}]$ with $V(H)=V(K_{r-1}[c_{1},\dots,c_{r-1}])$, and, let $H_{1}, \ldots,$ and $H_{r-1}$ be the corresponding $r-1$ parts of $V(H)$. We further assume there is a set $S$ of $\alpha n$ vertices in $H$ spanning at most $\lfloor \beta n^{2}\rfloor$ edges, and $H$ contains all the other edges in $K_{r-1}[c_{1},\dots,c_{r-1}]$ incident to any vertex of $V(H)\setminus S$. We shall find a suitable $H$ in $\mathcal{G}^*$ with some nested structure below.

   \begin{claim}\label{claim:well-behaved H in G*}
    There is an $(r-1)$-partite graph $H$ in $\mathcal{G}^*$ containing a set $S$ of $\alpha n$ vertices that spans at most $\lfloor \beta n^{2}\rfloor$ edges and satisfies that there exists an integer $s\in [r-1]$ such that 
    \begin{enumerate}[label=(\roman*)]
    \setlength{\itemsep}{0em}
        \item\label{H1} $\bigcup\limits_{i=1}^{s-1} H_i \subsetneq S\subseteq \bigcup\limits_{i=1}^s H_i$,
        \item\label{H2} for $i,j\in [r-1]\setminus [s]$, $\left||H_i| - |H_j|\right|\leq 1$,
        \item\label{H3} for $i,j\in [s-1]$, $\left||H_i| - |H_j|\right|\leq 1$.
    \end{enumerate}
   \end{claim}

    \begin{proof*}[Proof of~\cref*{claim:well-behaved H in G*}]
    For any $(r-1)$-partite graph $H \in \mathcal{G}^*$, let $\widehat{H}$ be  the complete $(r-1)$-partite graph whose parts coincide with those of $H$, and let $S_{\widehat{H}}$ be a sparsest $\alpha n$-vertex subset of $\widehat{H}$. Let $H_1, \dots, H_{r-1}$ denote the vertex parts of $H$ and $|H_1|\geq |H_2|\geq \dots\geq |H_{r-1}|$.

    We first claim that for any $i\in [r-1]$ except at most one such index, either $H_i\subseteq S_{\widehat{H}}$ or $S_{\widehat{H}}\cap H_i=\emptyset$. Otherwise, there exist distinct $i,j$ satisfying both $0<|H_i\cap S_{\widehat{H}}|<|H_i|$ and $0<|H_j\cap S_{\widehat{H}}|<|H_j|$. Assume $i<j$ and thus $|H_i|\geq |H_j|$. Then in $S_{\widehat{H}}$ we replace an $\ell$-vertex subset of $H_j \cap S_{\widehat{H}}$ by an $\ell$-vertex subset of $H_i \setminus S_{\widehat{H}}$, where $0<\ell = \min\{ |H_i \setminus S_{\widehat{H}}|,|H_j \cap S_{\widehat{H}}|\}$. We obtain an $\alpha n$-vertex subset of $\widehat{H}$ that is sparser than $S_{\widehat{H}}$, a contradiction. Thus, for every $\widehat{H}$, the sparsest $\alpha n$-vertex subset $S_{\widehat{H}}$ contains as many entire parts as possible, and after justifying the index of $H_i$'s if necessary, there is $s \in [r-1]$ such that $\bigcup_{i=1}^{s-1} H_i \subsetneq S_{\widehat{H}} \subseteq \bigcup_{i=1}^{s} H_i$.
    
    If $e_{\widehat{H}}(S_{\widehat{H}})\leq \lfloor \beta n^2 \rfloor$, then as $H\subset \widehat{H}$, $e(H)\leq e(\widehat{H})$ and thus $\widehat{H}\in \mathcal{G}^*$. So $\widehat{H}$ and $S_{\widehat{H}}$ satisfy \ref{H1}. If $e_{\widehat{H}}(S_{\widehat{H}})>\lfloor \beta n^2 \rfloor$, then let $H'$ denote a graph obtained from $\widehat{H}$ by deleting $e_{\widehat{H}}(S_{\widehat{H}})-\lfloor \beta n^2 \rfloor$ edges from $E(S_{\widehat{H}})$. Note that $H'\subseteq \widehat{H}$ and $H'$ contains an $\alpha n$-vertex subset $S_{H'}$ (the same vertex set as $S_{\widehat{H}}$) with $e_{H'}(S_{H'})=\lfloor \beta n^2 \rfloor$. Moreover, we consider a sparsest $\alpha n$-vertex subset $S_H$ of the original graph $H$ (which can be also regarded as a vertex subset of $\widehat{H}$) and thus $e_{H}(S_H)\leq \lfloor \beta n^2 \rfloor$. Let $H^*$ denote a graph obtained from $\widehat{H}$ by deleting $e_{\widehat{H}}(S_H)-e_{H}(S_H)$ edges from $E(S_H)$. Since $S_{\widehat{H}}$ is a sparsest subset of the complete $(r-1)$-partite graph $\widehat{H}$, $e_{\widehat{H}}(S_{\widehat{H}})-\lfloor \beta n^2 \rfloor\leq e_{\widehat{H}}(S_{H})-\lfloor \beta n^2 \rfloor \leq e_{\widehat{H}}(S_{H})-e_{H}(S_H)$, and thus $e(H)\leq e(H^*)\leq e(H')$. Therefore, $H' \in \mathcal{G}^*$ and for some $s \in [r-1]$, $\bigcup_{i=1}^{s-1} H'_i \subsetneq S_{H'} \subseteq \bigcup_{i=1}^{s} H'_i$, where $H'_i$ corresponds to the part $H_i$ when restricted to $H'$. So $H'$ and $S_{H'}$ satisfy~\ref{H1}.

    Next we consider $H'\in \mathcal{G}^*$ and its sparse $\alpha n$-vertex subset $S_{H'}$. For distinct $i,j$ such that $i,j\in [r-1]\setminus [s]$, assume that $|H'_i|<|H'_j|$. We apply the following operation: delete a set $X\subseteq H'_j$ of size $\left\lfloor\frac{|H'_j|-|H'_i|}{2}\right\rfloor$, and add a set $Y$ of the same size to $H'_i$ (while establishing a one-to-one correspondence between the elements of $X$ and $Y$), and for each $y \in Y$ and its corresponding element $x \in X$, let $y$ be adjacent to $|N_{H'}(x) \cap H'_i|$ vertices in $H'_j \setminus X$ and to all other neighbors of $x$ in $V(H') \setminus (H'_i \cup H'_j)$.  This yields a new $(r-1)$-partite graph $H''$ satisfying $e(H'') \ge e(H')$ and $\left||H''_i| - |H''_j|\right|\leq 1$. Consequently, $\left||H''_i| - |H''_j|\right|\leq 1$ for all $i,j \in [r-1]\setminus[s]$. Furthermore, $H''$ and the set $S_{H''}$ (which coincides with $S_{H'}$) satisfy~\ref{H1}. Similarly, by repeatedly applying the same operation we may ensure that $\left||H''_i| - |H''_j|\right|\leq1$ for every distinct $i,j \in [s-1]$.  Since this operation does not increase the number of edges induced by $S_{H''}$, the sparsity of the selected subset is preserved. Hence, the resulting $H''$ and $S_{H''}$ satisfy Conditions~\ref{H2} and~\ref{H3}.
    \end{proof*}
    
   From now on, we consider a graph $H \in \mathcal{G}^*$ and its sparse $\alpha n$-vertex subset $S$ that satisfy all the conditions in \cref{claim:well-behaved H in G*}. We first complete two simple cases based on the range of $\alpha$.
   \begin{claim}\label{claim:case-alpha}
    The statement of~\cref{thm-inverse problem of local density} holds when $0\leq\alpha\leq \frac{1}{r-1}+d(2)$ and $\frac{r-2}{r-1}+d(r-1)\leq\alpha\leq 1$. Moreover, $s\notin\{1,r-1\}$.
   \end{claim}
   \begin{proof*}[Proof of~\cref*{claim:case-alpha}]
   If $0\leq\alpha\leq \frac{1}{r-1}+d(2)$, then it is straightforward to see that $\beta=0$ by~\cref{thm-upper bound on Ks-independence number}. We now prove that the case when $\frac{r-2}{r-1}+d(r-1)\leq\alpha\leq 1$ and $\beta=\gamma-(1-\alpha)(\frac{r-2}{r-1}+d(r-1))$. If $\Delta(H)>(\frac{r-2}{r-1}+d(r-1))n$, then by~\cref{thm-upper bound on Ks-independence number} we have $e(G)<\gamma n^2$ as desired. So we may assume $\Delta(H)\leq (\frac{r-2}{r-1}+d(r-1))n$. Since $e(S^{c})+e(S^{c},S)\leq \sum\limits_{w\in S^{c}}d(w)\leq (1-\alpha)(\frac{r-2}{r-1}+d(r-1))n^{2}$, the value of $\beta$ is obtained.

   We then prove the moreover part. When $s=1$, $S\subseteq H_1$, and hence $S$ is independent. By~\cref{thm-upper bound on Ks-independence number}, we may assume $|S|=\alpha n\leq (\frac{1}{r-1}+d(2))n$. Thus the case $0\leq\alpha\leq\frac{1}{r-1}+d(2)$ has already been settled. When $s=r-1$, the set $S$ contains $\cup_{i=1}^{r-2}H_i$, the union of the largest $r-2$ parts of $G$. So $|S|\geq \sum_{i=1}^{r-2}|H_i|\geq \Delta(H)$. Since we may assume that $e(H)\geq\gamma n^2$, \cref{thm-upper bound on Ks-independence number}, together with the assumption $e(S)\leq \beta n^2$, implies that $\Delta(H)\geq (\frac{r-2}{r-1}+d(r-1))n$. Indeed, otherwise we would have $e(H)\leq e(S)+|S^c|\Delta(H)\leq \gamma n^2$ as desired. Consequently, $|S|=\alpha n\geq (\frac{r-2}{r-1}+d(r-1))n$. The case $\frac{r-2}{r-1}+d(r-1) \leq\alpha\leq1$ has already been proved above.
   \end{proof*}

   Let $x_{1} n:= \alpha n-(|H_{1}|+\dots+|H_{s-1}|)$ and $x_{2} n:=|H_{s}|-x_{1}n$. We have that
    \begin{align*}
    e(H)&=e(S)+e(S,S^{c})+e(S^{c})\\
        &\leq \beta n^{2}+x_{2}(1-x_{1}-x_{2})n^{2}+\alpha(1-\alpha-x_{2})n^{2}+\frac{r-s-2}{2(r-s-1)}(1-\alpha-x_{2})^{2}n^{2}\\
        &=:f(x_{1},x_{2},s) n^{2}.
    \end{align*}
    
   It suffices to prove that $f(x_1,x_2,s)\leq \gamma$ for all real numbers $x_1$ and $x_2$ satisfying the following constraints:
    \begin{equation}\label{ineq-conditions of x1 and x2}
        \left\{\begin{array}{lll}
              x_{1}+x_{2}&\geq &\dfrac{1-\alpha-x_{2}}{r-s-1}, \\[8pt]
              x_{1}+x_{2}&\leq & \dfrac{\alpha-x_{1}}{s-1}, \\[8pt]
              \alpha-x_{1}&\leq & \dfrac{s-1}{r-1}+d(s).
        \end{array}
        \right.
    \end{equation}
    Here, the first inequality is equivalent to the condition $|H_{s}|\geq \frac{1}{r-s-1}(|H_{s+1}|+\dots+|H_{r-1}|)$; the second inequality says that $|H_{s}|\leq \frac{1}{s-1}(|H_{1}|+\dots+|H_{s-1}|)$; the third inequality reflects the fact that the maximum $K_{s}$-free subgraph in $H$ has size at most $(\frac{s-1}{r-1}+d(s))n$, as otherwise we have $e(H)< \gamma n^{2}$ by~\cref{thm-upper bound on Ks-independence number}. See~\cref{fig:calculation} for an explanation. We remark that the bounds in the first two conditions of Constraints~\eqref{ineq-conditions of x1 and x2} are valid only under the assumption that $s\notin\{1,r-1\}$, which is already guaranteed by~\cref{claim:case-alpha}.

\begin{figure}[h]
\centering
\begin{tikzpicture}[
    scale=0.42,
    cluster/.style={
        ellipse,
        draw=black,
        line width=0.9pt,
        minimum width=5cm,
        minimum height=4cm,
        inner sep=0pt,
        transform shape
    },
    smallcluster/.style={
        circle,
        draw=black,
        line width=0.9pt,
        minimum size=3cm,
        inner sep=0pt,
        transform shape
    },
    thick edge/.style={
        ultra thick,
        line cap=round,
        shorten <=2pt,
        shorten >=2pt
    },
    split/.style={
        purple,
        thick
    }
]

\node[cluster] (H1) at (-6,3.5) {};
\node[cluster] (H2) at (-6,-4) {};
\node[cluster] (Hs) at (0,-0.5) {};
\node[smallcluster] (Hr) at (6.5,-0.5) {};

\node at (-6,3.5) {$H_{1}$};
\node at (-6,-4) {$H_{2}$};
\node at (1,3) {$H_{s}$};
\node at (6.5,-0.5) {$H_{r-1}$};
\node at (-1.05,-0.5) {$x_{1}n$};
\node at (1.45,-0.5) {$x_{2}n$};

\draw[split] (-0.2,1.43) arc[start angle=40, end angle=-40, radius=3];

\draw[thick edge] (H1.south east) -- (Hs.north west);
\draw[thick edge] (H2.north east) -- (Hs.south west);
\draw[thick edge] (H1.south) -- (H2.north);
\draw[thick edge] (H1.east) -- (Hr.north west);
\draw[thick edge] (H2.east) -- (Hr.south west);
\draw[thick edge] (Hs.east) -- (Hr.west);
\end{tikzpicture}
    \caption{The structure of $H$}
    \label{fig:calculation}
    \end{figure}
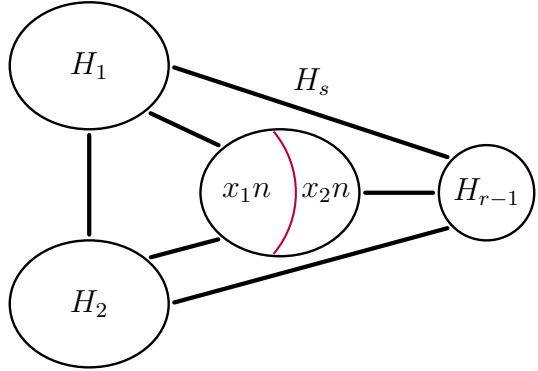
    
    We now establish some basic properties of $f(x_1,x_2,s)$.

    \begin{claim}\label{claim:f}
    The following statements hold.
    \begin{enumerate}[label=(\roman*)]
    \setlength{\itemsep}{0em}
        \item\label{f-property1} $f(x_{1},x_{2},s)$ is monotonically increasing as $x_{1}$ decreases;
        \item\label{f-property2} $f(x_{1},x_{2},s)$ is monotonically increasing as $x_{2}$ decreases;
        \item\label{f-property3} If $x_{1}+x_{2}=\frac{1-\alpha-x_2}{r-s-1}$ and $x_1\leq \frac{1-\alpha}{r-s-1}$, then $f(x_{1},x_{2},s)$ is monotonically increasing when $x_{1}$ decreases.
    \end{enumerate}
    \end{claim}

    \begin{proof*}[Proof of~\cref*{claim:f}]
    Note that~\ref{f-property1} is clear. For~\ref{f-property2}, we consider the partial derivative of $f(x_{1},x_{2},s)$. We have that
    \[
    \frac{\partial f}{\partial x_{2}}=1-x_{1}-2x_{2}-\alpha+\frac{r-s-2}{r-s-1}(x_{2}+\alpha-1)=-(x_{1}+x_{2})+\frac{1-\alpha-x_{2}}{r-s-1}\leq 0,
    \]
    where the last inequality follows from the first constraint of \eqref{ineq-conditions of x1 and x2}. This proves~\ref{f-property2}.
    
    For \ref{f-property3}, we consider the partial derivative of $f(x_{1},x_{2},s)$. We have that
    {\small   
    \begin{align*}
    \frac{\partial f}{\partial x_{1}}&=-\frac{r-s-1}{r-s}(1-x_1-x_2)-x_{2}(1-\frac{r-s-1}{r-s})+\frac{r-s-1}{r-s}\alpha-\frac{r-s-2}{r-s}(x_{2}+\alpha-1)\\
    &=\frac{r-s-1}{r-s}x_{1}-\frac{1-\alpha}{r-s}\leq 0,
    \end{align*}}
    where the last inequality follows from $x_1\leq \frac{1-\alpha}{r-s-1}$. Then we have $f(x_{1},x_{2},s)$ is monotonically increasing as $x_{1}$ decreases.
    \end{proof*}
    
    By~\cref{claim:case-alpha}, we now focus on the remaining cases based on the range of $\alpha$. For each integer $k$ with $2\leq k\leq r-2$, we consider the ranges $\frac{k-1}{r-1}+d(k)<\alpha\leq\frac{k}{r-1}+\frac{r-k-1}{r-k}d(k)$ and $\frac{k}{r-1}+\frac{r-k-1}{r-k}d(k)<\alpha\leq\frac{k}{r-1}+d(k+1)$, respectively. We need the following claim.

    \begin{claim}\label{claim:not intersecting parts}
    For $\frac{k-1}{r-1}+d(k)<\alpha\leq\frac{k}{r-1}+d(k+1)$, the number of parts disjoint from $S$ is at most $r-k-1$, i.e., $s\geq k$.
    \end{claim}

    \begin{proof*}[Proof of~\cref*{claim:not intersecting parts}]
    Otherwise, we suppose that the number of parts disjoint from $S$ is at least $r-k$, which implies that the number of parts intersecting with $S$ is at most $k-1$ parts. But their union has size at least $\alpha n>(\frac{k-1}{r-1}+d(k))n$ by the assumption,  contradicting the upper bound of $K_{s}$-independence number of $H$ shown in~\cref{thm-upper bound on Ks-independence number}. 
    \end{proof*}

    Let $k$ be an integer such that $2\leq k\leq r-2$. To complete the proof, it remains to consider the following two cases.

    \medskip
    \noindent\textbf{Case 1.} \emph{$\frac{k-1}{r-1}+d(k)\leq\alpha\leq \frac{k}{r-1}+\frac{r-k-1}{r-k}d(k)$} \emph{and $\beta=\frac{k-2}{2(k-1)}\left(\frac{k-1}{r-1}+d(k)\right)^{2}+\left(\frac{k-1}{r-1}+d(k)\right)\left(\alpha-\frac{k-1}{r-1}-d(k)\right)$.}
    \medskip
    
    By~\cref{claim:not intersecting parts}, we consider the following two subcases. 
    
    If there are exactly $r-k-1$ parts disjoint from $S$, then by~\cref{claim:f}\ref{f-property2} and Constraints~\eqref{ineq-conditions of x1 and x2} with $s=k$, we can decrease $x_{2}$ until $x_{1}+x_{2}=\frac{1-\alpha-x_{2}}{r-k-1}$, or $x_{2}=0$ and $x_{1}>\frac{1-\alpha-x_{2}}{r-k-1}$.
    If $x_{1}+x_{2}=\frac{1-\alpha-x_{2}}{r-k-1}$, then since $x_{2}\geq 0$, we have $x_{1}\leq \frac{1-\alpha}{r-k-1}$. By~\cref{claim:f}\ref{f-property3}, we can assume that $x_{1}$ is as small as possible with $x_{1}+x_{2}=\frac{1-\alpha-x_{2}}{r-k-1}$. Recalling that $x_{1}\geq \alpha-\frac{k-1}{r-1}-d(k)$, so we have
    \begin{align*}
        f(x_{1},x_{2},s)\leq f(\alpha-\frac{k-1}{r-1}-d(k),\, \frac{k}{r-1}+\frac{r-k-1}{r-k}d(k)-\alpha,\, k)= \gamma.
    \end{align*}
    If $x_{2}=0$ and $x_{1}\geq\frac{1-\alpha-x_{2}}{r-k-1}=\frac{1-\alpha}{r-k-1}$, then by~\cref{claim:f}\ref{f-property1}, we can decrease $x_{1}$ until $x_{1}=\frac{1-\alpha}{r-k-1}$, and now we have that 
    {\small 
    \begin{align}\label{ineq-s=k}
    f(x_{1},x_{2},s)
    &\leq f\left(\frac{1-\alpha}{r-k-1},0,k\right) \notag\\
    &=\beta+\alpha(1-\alpha)+\frac{r-k-2}{2(r-k-1)}(1-\alpha)^{2} \\
    &=\frac{k-2}{2(k-1)}
    \left(\frac{k-1}{r-1}+d(k)\right)^{2}+\left(\frac{k-1}{r-1}+d(k)\right)
    \left(\alpha-\frac{k-1}{r-1}-d(k)\right)
    +\alpha(1-\alpha) \notag\\
    &\quad+\frac{r-k-2}{2(r-k-1)}(1-\alpha)^{2} \notag\\
    &=\gamma+\left(\frac{k-1}{r-1}+d(k)\right)
    \left(\alpha-\frac{k-1}{r-1}-d(k)\right)+\alpha(1-\alpha)
    +\frac{r-k-2}{2(r-k-1)}(1-\alpha)^{2} \notag\\
    &\quad-\left(\frac{k-1}{r-1}+d(k)\right)
    \left(\frac{r-k}{r-1}-d(k)\right) -\frac{r-k-1}{2(r-k)}
    \left(\frac{r-k}{r-1}-d(k)\right)^2
    \notag\\
    &\leq \gamma, \notag
    \end{align}
    }
    where the last inequality follows from the fact that the formula in the second last line is a quadratic polynomial in $\alpha$, which reaches its maximum value $\gamma$ in $\alpha = \frac{k}{r-1} + \frac{r-k-1}{r-k}d(k)$.

    If there are at most $r-k-2$ parts disjoint from $S$, which is equivalent to the case that there are at least $k+1$ parts intersecting with $S$, then now $s\geq k+1$. If $s\geq k+2$, then $S$ contains $\cup_{i=1}^{k+1}H_i$ which is the union of the largest $k+1$ parts of $G$, and thus by~\eqref{ineq-s=k},
    \begin{align*}
        e(G)&\leq e(S)+e(S,S^c)+e(S^c)\\
        &\leq \beta n^2+\alpha(1-\alpha)n^2+\frac{r-k-3}{2(r-k-2)}(1-\alpha)^2n^2\\
        &\leq \beta n^2+\alpha(1-\alpha) n^2+\frac{r-k-2}{2(r-k-1)}(1-\alpha)^{2} n^2\\
        &=f\left(\frac{1-\alpha}{r-k-1},0,k\right)n^2\\
        &\leq \gamma n^2.
    \end{align*} 
    We are done.
    Thus, we assume that $s=k+1$.
    It follows from the condition of $\alpha\leq \frac{k}{r-1}+d(k+1)$ and Constraints~\eqref{ineq-conditions of x1 and x2} that we can decrease $x_{1}$ until $x_{1}+x_{2}=\frac{1-\alpha-x_{2}}{r-k-2}$. Note that during this process, Constraints~\eqref{ineq-conditions of x1 and x2} always hold. And then by~\cref{claim:f}\ref{f-property3} with $x_{1}\leq \frac{1-\alpha}{r-k-2}$, we can fix $x_{1}+x_{2}=\frac{1-\alpha-x_{2}}{r-k-2}$ and decrease $x_{1}$ until $x_{1}=0$. Now we have $x_{2}=\frac{1-\alpha}{r-k-1}$, and
    \[
    f(x_{1},x_{2},s)\leq f(0,\frac{1-\alpha}{r-k-1},k+1)=f(\frac{\alpha}{k},0,k),
    \]
    which goes back to the case where there are exactly $r-k-1$ parts disjoint from $S$, and we complete the proof of Case 1.

    \medskip
    \noindent\textbf{Case 2.} \emph{$\frac{k}{r-1}+\frac{r-k-1}{r-k}d(k)\leq\alpha\leq \frac{k}{r-1}+d(k+1)$ and  $\beta=\gamma-\frac{r-k-2}{2(r-k-1)}(1-\alpha)^{2}-\alpha(1-\alpha).$}

    Again by~\cref{claim:not intersecting parts}, we consider the following two subcases.

    If there are at most $r-k-2$ parts disjoint from $S$, since $\alpha\leq \frac{k}{r-1}+d(k+1)$, then by~\cref{claim:f}\ref{f-property1} and \ref{f-property3} and similar to the former case, we have 
    \[
    f(x_{1},x_{2},k+1)\leq f(0,\frac{1-\alpha}{r-k-1},k+1)\leq \gamma.
    \]
    If there are exactly $r-k-1$ parts that do not intersect with $S$, since $x_{1}\geq \alpha-\frac{k-1}{r-1}-d(k)\geq \frac{1-\alpha}{r-k-1}$, then by~\cref{claim:f}\ref{f-property2} we can decrease $x_{2}$ until $x_{2}=0$ and we have 
    \[
    f(x_{1},x_{2},k)\leq f(x_{1},0,k)\leq \gamma.
    \]
    The last inequality holds since the function $f$ is independent of $x_{1}$ when $x_{2}=0$. We complete the proof of Case 2.
\end{proof}

\section[RegularityLemma]{Local clique density in $H$-free graphs}

In this section, we establish the extension of \cref{thm-inverse problem of local density} to $H$-free graphs. Our proof of~\cref{thm-H-free graph case} uses the Regularity Lemma and we refer the interested reader to the survey \cite{application of regularity lemma}.

We begin with some necessary definitions. Let $G=(V,E)$ be a graph, and let $A, B$ be two disjoint vertex subsets of $V(G)$. For non-empty $A$ and $B$, we define the \emph{density} of edges between $A$ and $B$ by 
\[
d(A,B)=\frac{e(A,B)}{|A||B|}.
\]
For $\varepsilon>0$, we say the pair $(A,B)$ is \emph{$\varepsilon$-regular} if for every $X\subseteq A$ and $Y\subseteq B$ with $|X|>\varepsilon |A|$ and $|Y|>\varepsilon |B|$, we have
\[
|d(X,Y)-d(A,B)|<\varepsilon.
\]
We say a partition $V=V_1\cup \dots\cup V_{k}$ is \emph{balanced} if for all $i,j\in [k]$, $\big||V_i|-|V_j|\big|\leq 1$. A balanced partition $V=V_1\cup \dots\cup V_{k}$ is called \emph{$\varepsilon$-regular} if $|V_i|\leq \varepsilon|V|$ for every $i\in [k]$ and all but at most $\varepsilon k^2$ of the pairs $(V_i,V_j)$ are $\varepsilon$-regular. The above partition is called \emph{totally $\varepsilon$-regular} if all the pairs $(V_i,V_j)$ are $\varepsilon$-regular. The following celebrated lemma was proved by Szemer\'edi~\cite{1978-szemeredi}.
\begin{lemma}[Szemer\'edi Regularity Lemma~\cite{1978-szemeredi}]\label{lem-regularity lemma}
    For every $\varepsilon>0$, there is an integer $M(\varepsilon)$ such that for sufficiently larger $n$ every $n$-vertex graph has an $\varepsilon$-regular balanced partition into $k$ classes, where $\frac{1}{\varepsilon} \leq k\leq M(\varepsilon)$.
\end{lemma}
When applying Szemer\'edi Regularity Lemma, we shall consider the corresponding reduced graph and we need the following lemma, which appears as Theorem~2.1 in \cite{application of regularity lemma}.
\begin{lemma}{\em \cite{application of regularity lemma}}\label{lem-embedding lemma}
    For every $\delta>0$ and integers $m$ and $\ell$, there exist $\varepsilon=\varepsilon(\delta,\ell,m)>0$ and $n_0=n_0(\delta,t,m)$ with the following property: If $G$ is an $n$-vertex graph with $n>n_0$ and $(V_1,\dots,V_m)$ is a totally $\varepsilon$-regular partition of $V(G)$ such that $d(V_i,V_j)\geq \delta$ for all $i<j$, then $G$ contains a complete $m$-partite subgraph $K_{m}(\ell)$ (that is, an $\ell$-blow-up of the complete graph $K_m$).
\end{lemma}
Now we are ready to prove \cref{thm-H-free graph case}.
\begin{proof}[Proof of \cref*{thm-H-free graph case}]
    Let $\ell$ be the smallest integer such that $H\subseteq K_{r}(\ell)$. Such an integer exists since $H$ is an $r$-chromatic graph.
    For a small constant $\delta>0$, there are $\varepsilon=\min\{\delta,\varepsilon(\delta,\ell,r)\}$ and $M(\varepsilon)$ such that $\varepsilon(\delta,\ell,r)$ is obtained in~\cref{lem-embedding lemma} and $M(\varepsilon)$ is obtained in~\cref{lem-regularity lemma}. 
    Assume that $n$ is a sufficiently large integer and let $G$ be an $n$-vertex $H$-free graph with at least $\gamma n^2$ edges. By~\cref{lem-regularity lemma}, the graph $G$ admits an $\varepsilon$-regular balanced partition $V(G)=V_1\cup \dots\cup V_k$.

    We first define the reduced graph $R$ with respect to the $\varepsilon$-regular balanced partition $V(G)=V_1\cup \dots\cup V_k$. Set $V(R)=\{v_1,\dots,v_k\}$. For two distinct vertices $v_i$ and $v_j$, $v_i v_j$ is an edge of $R$ if and only if $(V_i,V_j)$ is an $\varepsilon$-regular pair with $d(V_i,V_j)\geq \delta$. Since $G$ is $H$-free, by~\cref{lem-embedding lemma}, $R$ is also $K_r$-free. Otherwise, $R$ contains $K_r$ as an induced subgraph, and it follows from~\cref{lem-embedding lemma} that we can find a $K_{r}(\ell)$ in $G$ and thus find a copy of $H$, which is a contradiction.

    Then we construct a graph $G'$ from $G$ by deleting three parts of edges as follows:
    \begin{enumerate}[label=(\arabic*)]
    \setlength{\itemsep}{0mm}
        \item Delete all edges inside each $V_i$. (This removes at most $\frac{n^2}{k}$ edges.)
        \item Delete all edges between $V_i$ and $V_j$ whenever $(V_i,V_j)$ is not an $\varepsilon$-regular pair. (This removes at most $\varepsilon k^2(\frac{n}{k})^2 =\varepsilon n^2$ edges.)
        \item Delete all edges between $V_i$ and $V_j$ whenever $d(V_i,V_j)<\delta$. (This removes at most $\delta n^2$ edges.)
    \end{enumerate}
    Hence, we have that
    \[
    e(G')\geq e(G)-(\frac{1}{k}+\varepsilon+\delta)n^2\geq e(G)-3\delta n^2\geq (\gamma-3\delta)n^2.
    \]
    We claim that $G'$ is $K_r$-free. Since in $G'$, there are edges between $V_i$ and $V_j$ if and only if $(V_i,V_j)$ is an $\varepsilon$-regular pair with $d(V_i,V_j)\geq \delta$. If $G'$ contains a $K_r$, then $R$ contains $K_r$ as a subgraph by the definition of $R$, a contradiction.

    Therefore, we can apply~\cref{thm-inverse problem of local density} to $G'$. Noting that $e(G')\geq (\gamma-3\delta)n^2$, there is a constant $\beta'=\beta(r,t,\alpha,\gamma-3\delta)$ such that every $\alpha n$ vertices in $G'$ contain at least $\beta' n^t$ copies of $K_t$. Since $\beta(r,t,\alpha,\gamma)$ is continuous in $\gamma$, and since $\delta>0$ is arbitrary, we may let $\delta\to0$ to obtain $\beta'=\beta-o(1)$, as desired.
\end{proof}

\section{Concluding remarks}\label{section-concluding remarks}

Our work provides a complete solution to \cref{prob-local density-inverse problem}, culminating in \cref{thm-inverse problem of local density}. The framework developed here focuses on lower-order cliques. Its extension and applications to $K_{r}$-free graphs for $r \geq 4$ constitute a natural and promising direction for further study. The analysis of the tight examples demonstrates that the highly unbalanced structure is a key limitation of \cref{thm-inverse problem of local density}, and imposing additional structural constraints may lead to significantly broader applications.

As remarked in~\cref{sec:intro}, to obtain a sparse subset of a prescribed size, it suffices to first find a sparse subset whose size is substantially smaller than the desired one by~\cref{thm-inverse problem of local density}. This relaxation leaves a larger remaining subgraph, which may provide more flexibility to find additional sparse subsets. This approach is particularly helpful when studying partition problems or the existence of multiple disjoint sparse subsets of a prescribed size. For example, finding two disjoint sparse subsets of size $\frac{n}{2}$ is equivalent to finding a balanced $2$-partition. Therefore, \cref{thm-inverse problem of local density} may serve as a useful tool for these problems. Related questions have been studied in \cite{2007-sudakov,2024-balogh}.

Very recently, Ma, Wang, Zhu \cite{MWZ2026}, and independently, Li, Liu, and Zhang~\cite{LLZ2026} determined the asymptotically sharp clique-to-clique density function. It would be interesting to extend our result to a local version of the clique-to-clique density theorem in $H$-free graphs.

\section*{Acknowledgments}
Jiaao Li and Xinyuan Li are partially supported by National Key Research and Development Program of China (No. 2022YFA1006400), National Natural Science Foundation of China (No. 12571371), Natural Science Foundation of Tianjin (No. 24JCJQJC00130), and the Fundamental Research Funds for the Central Universities, Nankai University. 
Yan Wang is supported by National Key R\&D Program of China (No. 2022YFA1006400) and National Natural Science Foundation of China (No. 12571376). 
Zhouningxin Wang is supported by National Natural Science Foundation of China (No. 12301444) and the Fundamental Research Funds for the Central Universities, Nankai University.

\end{document}